\documentclass[]{amsart}

\usepackage{amssymb}
\usepackage{upgreek}
\usepackage{amsmath}
\usepackage[all,cmtip]{xy}
\usepackage{enumitem}

\usepackage{amsfonts}
\usepackage{mathrsfs}
\usepackage{latexsym}
\usepackage{graphicx}
\usepackage{amscd,amssymb,amsmath,amsbsy,amsthm}
\usepackage[colorlinks,plainpages,backref,urlcolor=blue]{hyperref}
\allowdisplaybreaks

\usepackage{tikz}
\usetikzlibrary{arrows,calc}
\tikzset{
>=stealth',
help lines/.style={dashed, thick},
axis/.style={<->},
important line/.style={thick},
connection/.style={thick, dotted},
}

\newcommand{\nc}{\newcommand}
\nc{\rnc}{\renewcommand}
\nc{\bb}[1]{{\mathbb #1}}
\nc{\bbA}{\bb{A}}\nc{\bbB}{\bb{B}}\nc{\bbC}{\bb{C}}\nc{\bbD}{\bb{D}}
\nc{\bbE}{\bb{E}}\nc{\bbF}{\bb{F}}\nc{\bbG}{\bb{G}}\nc{\bbH}{\bb{H}}
\nc{\bbI}{\bb{I}}\nc{\bbJ}{\bb{J}}\nc{\bbK}{\bb{K}}\nc{\bbL}{\bb{L}}
\nc{\bbM}{\bb{M}}\nc{\bbN}{\bb{N}}\nc{\bbO}{\bb{O}}\nc{\bbP}{\bb{P}}
\nc{\bbQ}{\bb{Q}}\nc{\bbR}{\bb{R}}\nc{\bbS}{\bb{S}}\nc{\bbT}{\bb{T}}
\nc{\bbU}{\bb{U}}\nc{\bbV}{\bb{V}}\nc{\bbW}{\bb{W}}\nc{\bbX}{\bb{X}}
\nc{\bbY}{\bb{Y}}\nc{\bbZ}{\bb{Z}}
\nc{\mbf}[1]{{\mathbf #1}}
\nc{\bfA}{\mbf{A}}\nc{\bfB}{\mbf{B}}\nc{\bfC}{\mbf{C}}\nc{\bfD}{\mbf{D}}
\nc{\bfE}{\mbf{E}}\nc{\bfF}{\mbf{F}}\nc{\bfG}{\mbf{G}}\nc{\bfH}{\mbf{H}}
\nc{\bfI}{\mbf{I}}\nc{\bfJ}{\mbf{J}}\nc{\bfK}{\mbf{K}}\nc{\bfL}{\mbf{L}}
\nc{\bfM}{\mbf{M}}\nc{\bfN}{\mbf{N}}\nc{\bfO}{\mbf{O}}\nc{\bfP}{\mbf{P}}
\nc{\bfQ}{\mbf{Q}}\nc{\bfR}{\mbf{R}}\nc{\bfS}{\mbf{S}}\nc{\bfT}{\mbf{T}}
\nc{\bfU}{\mbf{U}}\nc{\bfV}{\mbf{V}}\nc{\bfW}{\mbf{W}}\nc{\bfX}{\mbf{X}}
\nc{\bfY}{\mbf{Y}}\nc{\bfZ}{\mbf{Z}}
\nc{\bfa}{\mbf{a}}\nc{\bfb}{\mbf{b}}\nc{\bfc}{\mbf{c}}\nc{\bfd}{\mbf{d}}
\nc{\bfe}{\mbf{e}}\nc{\bff}{\mbf{f}}\nc{\bfg}{\mbf{g}}\nc{\bfh}{\mbf{h}}
\nc{\bfi}{\mbf{i}}\nc{\bfj}{\mbf{j}}\nc{\bfk}{\mbf{k}}\nc{\bfl}{\mbf{l}}
\nc{\bfm}{\mbf{m}}\nc{\bfn}{\mbf{n}}\nc{\bfo}{\mbf{o}}\nc{\bfp}{\mbf{p}}
\nc{\bfq}{\mbf{q}}\nc{\bfr}{\mbf{r}}\nc{\bfs}{\mbf{s}}\nc{\bft}{\mbf{t}}
\nc{\bfu}{\mbf{u}}\nc{\bfv}{\mbf{v}}\nc{\bfw}{\mbf{w}}\nc{\bfx}{\mbf{x}}
\nc{\bfy}{\mbf{y}}\nc{\bfz}{\mbf{z}}

\nc{\mcal}[1]{{\mathcal #1}}
\nc{\calA}{\mcal{A}}\nc{\calB}{\mcal{B}}\nc{\calC}{\mcal{C}}\nc{\calD}{\mcal{D}}
\nc{\calE}{\mcal{E}} \nc{\calF}{\mcal{F}}\nc{\calG}{\mcal{G}}\nc{\calH}{\mcal{H}}
\nc{\calI}{\mcal{I}}\nc{\calJ}{\mcal{J}}\nc{\calK}{\mcal{K}}\nc{\calL}{\mcal{L}}
\nc{\calM}{\mcal{M}}\nc{\calN}{\mcal{N}}\nc{\calO}{\mcal{O}}\nc{\calP}{\mcal{P}}
\nc{\calQ}{\mcal{Q}}\nc{\calR}{\mcal{R}}\nc{\calS}{\mcal{S}}\nc{\calT}{\mcal{T}}
\nc{\calU}{\mcal{U}}\nc{\calV}{\mcal{V}}\nc{\calW}{\mcal{W}}\nc{\calX}{\mcal{X}}
\nc{\calY}{\mcal{Y}}\nc{\calZ}{\mcal{Z}}
\nc{\fA}{\frak{A}}\nc{\fB}{\frak{B}}\nc{\fC}{\frak{C}} \nc{\fD}{\frak{D}}
\nc{\fE}{\frak{E}}\nc{\fF}{\frak{F}}\nc{\fG}{\frak{G}}\nc{\fH}{\frak{H}}
\nc{\fI}{\frak{I}}\nc{\fJ}{\frak{J}}\nc{\fK}{\frak{K}}\nc{\fL}{\frak{L}}
\nc{\fM}{\frak{M}}\nc{\fN}{\frak{N}}\nc{\fO}{\frak{O}}\nc{\fP}{\frak{P}}
\nc{\fQ}{\frak{Q}}\nc{\fR}{\frak{R}}\nc{\fS}{\frak{S}}\nc{\fT}{\frak{T}}
\nc{\fU}{\frak{U}}\nc{\fV}{\frak{V}}\nc{\fW}{\frak{W}}\nc{\fX}{\frak{X}}
\nc{\fY}{\frak{Y}}\nc{\fZ}{\frak{Z}}
\nc{\fa}{\frak{a}}\nc{\fb}{\frak{b}}\nc{\fc}{\frak{c}} \nc{\fd}{\frak{d}}
\nc{\fe}{\frak{e}}\nc{\fFf}{\frak{f}}\nc{\fg}{\frak{g}}\nc{\fh}{\frak{h}}
\nc{\fri}{\frak{i}}\nc{\fj}{\frak{j}}\nc{\fk}{\frak{k}}\nc{\fl}{\frak{l}}
\nc{\fm}{\frak{m}}\nc{\fn}{\frak{n}}\nc{\fo}{\frak{o}}\nc{\fp}{\frak{p}}
\nc{\fq}{\frak{q}}\nc{\fr}{\frak{r}}\nc{\fs}{\frak{s}}\nc{\ft}{\frak{t}}
\nc{\fu}{\frak{u}}\nc{\fv}{\frak{v}}\nc{\fw}{\frak{w}}\nc{\fx}{\frak{x}}
\nc{\fy}{\frak{y}}\nc{\fz}{\frak{z}}

\newtheorem{theorem}{Theorem}[section]
\newtheorem{lemma}[theorem]{Lemma}
\newtheorem{corollary}[theorem]{Corollary}

\theoremstyle{definition}
\newtheorem{definition}[theorem]{Definition}
\newtheorem{example}[theorem]{Example}

\theoremstyle{remark}
\newtheorem{remark}[theorem]{Remark}

\numberwithin{equation}{section}

\DeclareMathOperator{\im}{im}

\DeclareMathOperator{\Hom}{{Hom}}

\DeclareMathOperator{\Ell}{\mathcal{E}}
\DeclareMathOperator{\Pic}{Pic}

\newcommand{\cQ}{\mathcal {Q}}

\newcommand{\catA}{\mathfrak{A}}

\newcommand{\al}{\alpha}
\newcommand{\la}{\lambda}

\newcommand{\de}{\delta}

\newcommand{\dyn}{{\operatorname{d}}}

\DeclareMathOperator{\Kr}{{{Kr}}}

\newcommand{\unit}{\mathbf{1}}

\newcommand{\ep}{\epsilon}

\newcount\cols
{\catcode`,=\active\catcode`|=\active
 \gdef\Young(#1){\hbox{$\vcenter
 {\mathcode`,="8000\mathcode`|="8000
  \def,{\global\advance\cols by 1 &}%
  \def|{\cr
        \multispan{\the\cols}\hrulefill\cr
        &\global\cols=2 }%
  \offinterlineskip\everycr{}\tabskip=0pt
  \dimen0=\ht\strutbox \advance\dimen0 by \dp\strutbox
  \halign
   {\vrule height \ht\strutbox depth \dp\strutbox##
    &&\hbox to \dimen0{\hss$##$\hss}\vrule\cr
    \noalign{\hrule}&\global\cols=2 #1\crcr
    \multispan{\the\cols}\hrulefill\cr%
   }
 }$}}
}

\title[elliptic Schubert classes]
{Elliptic Schubert Classes and the Poincar\'e Duality}

\author{Changlong Zhong	}
\address{Hudson 209, 1400 Washington Ave, Albany, NY 12222}
\curraddr{}
\email{czhong@albany.edu}
\thanks{The author would like to thank C. Lenart and G. Zhao for the collaboration on elliptic Schubert classes, which lead to the present paper. He would also like to thank Rui Xiong for helpful discussions, and  the Institut des Hautes \'Etudes Scientifiques for its hospitality and support during his visit in July 2024. }

\date{}

\begin{document}
\maketitle
\begin{abstract}In this expository note, by using the Kostant-Kumar method,  we prove the Poincar\'e duality of the elliptic  classes associated to  Schubert varieties . 
\end{abstract}

\section{Introduction
}Equivariant elliptic cohomology has become a significant area of study in topology, with connections to mathematical physics, integrable systems, representation theory, and Schubert calculus. It complements rich theories like equivariant cohomology and equivariant K-theory. Recent progress has been made in the equivariant elliptic cohomology of symplectic resolutions, starting with \cite{AO21}. For example, the elliptic stable envelope has been studied and shown to have a close connection to 3d mirror symmetry in physics. In algebraic geometry, elliptic Schubert classes associated with Schubert varieties are defined in \cite{RW19, KRW20}. Indeed, it is known that the elliptic fundamental classes of Schubert varieties are not well defined due to the singularties. However, following Borisov-Libgober's work, Rim\'anyi-Weber were able to define these classes after introducing a new parameter known as the dynamical or K\"ahler parameter.  This parameter emerges from the action of the Weyl group of the Langlands dual root system. For example, it has been shown that the classical elliptic Demazure-Lusztig operators (without the dynamical parameter) do not satisfy the braid relations, whereas the modified operators, which include the dynamical parameter, do. This result is established through geometric methods in \cite{RW19}.

In \cite{RW19, LZZ23}, elliptic Schubert classes are defined and shown to be permuted by the elliptic Demazure-Lusztig operators. A natural next question is to find a Billey-AJS type formula for the restriction coefficients, along with the combinatorial properties of the elliptic Schubert classes. This will be the subject of joint work by the author, Lenart, and Xiong.

The torus fixed point method has been a key tool in studying equivariant cohomology theories. One of the most important approaches using this method is the one developed in Kostant and Kumar's groundbreaking work from 1986 and 1990 \cite{KK86, KK90}. In this paper, we develop the elliptic Schubert classes using the Kostant-Kumar method, along with the axiomatic approach introduced in \cite{CZZ15}. In simple terms, this method begins by studying the elliptic Demazure-Lusztig operators. The elliptic Schubert classes are then defined as the classes obtained by applying these operators.

Our main results include Theorems \ref{thm:mat} and \ref{thm:poincare}. The first theorem examines the transition matrix between the (dynamical) Weyl group and the elliptic DL operators, while the second proves the Poincar\'e duality between the elliptic Schubert classes and their opposite counterparts. Although a similar result was obtained in \cite{LZZ23} using a more conceptual approach, the method in this paper is more algebraic and concrete. This makes it a more practical framework for applications in combinatorics.

The note is organized as follows: In Section 1, we define the elliptic twisted group algebra and the elliptic DL operators, and examine their transition matrix under the Weyl group action. In Section 2, we define the dual of the elliptic twisted group algebra, an auxiliary object that still captures the dynamical Weyl group. Finally, in Section 3, we define the elliptic Schubert classes and show that they are Poincar\'e dual to their opposite counterparts.

\section{The elliptic Demazure-Lusztig operators}
In this section, we define the elliptic twisted group algebra together with    the elliptic DL operators.
\subsection{Elliptic curve}
Let $E$ be an elliptic curve over $\bbC$. We have isomorphisms 
\[
E\cong \bbC/(\bbZ+\bbZ\tau)\cong \bbC^*/q^\bbZ, 
\]
where $\tau\in \bbC$ such that $\im \tau>0$, and $q=e^{2\pi i\tau}$. Recall the Jacobi-Theta function
\[\vartheta(u)=\frac{1}{2\pi i}(u^{1/2}-u^{-1/2})\prod_{s>0}(1-q^su)(1-q^su^{-1})\cdot \prod_{s>0}(1-q^s)^{-2}, \quad u\in \bbC^*.\]
Since $|q|<1$, this series converges and defines a holomorphic function on a double cover of $\bbC^*$ and  vanishes of order 1 at $q^{\bbZ}$. Write $\theta(x)=\vartheta(e^{2\pi i x})$ with $x\in \bbC$.

Let $G$ be a reductive group with a Borel subgroup $B$ and a maximal torus $T$.
Let $\bbX^*(T)$ (resp. $\bbX_*(T)$) be the group of characters (resp. cocharacters) of the torus $T$. Let $\Sigma\subset \bbX^*(T)$ be the set of roots, $\Pi=\{\al_1,...,\al_n\}$ be the set of simple roots, and $W$ be the Weyl group. 
 Define $\catA=\bbX_*(T)\otimes E$, and $\catA^\vee=\bbX^*(T)\otimes E\cong \Pic^0(\catA)$, which is the dual abelian variety of $\catA$. 
Each $\mu\in \bbX^*(T)$ defines a map $\chi_\mu:\catA\to E$. Denote $z_\mu=\chi_\mu(z), z\in \catA$. Dually, each $\mu^\vee\in \bbX_*(T)$ defines a map $\chi_{\mu^\vee}:\catA^\vee\to E$, and denote $\la_{\mu^\vee}=\chi_{\mu^\vee}(\la), \la\in \catA^\vee$. 
Note that 
\[
\bbX_*(T)\otimes \bbC\cong \bbC^n, ~\bbX^*(T)\otimes \bbC\cong \bbC^n.
\]  and $\theta(z_\mu), \theta(\la_{\mu^\vee})$ are considered as meromorphic functions on complex spaces.

\subsection{The twisted group algebra}
We fix $\hbar\in \bbC$. 
Let $\calQ$ be the  field of  meromorphic functions on 
\[\bbX_*(T)\otimes \bbC\times \bbX^*(T)\otimes \bbC\cong \bbC^n\times \bbC^n.\]
The actions of $W$ on $\bbX_*(T)$ and on $\bbX^*(T)$ induce two commuting actions on $\cQ$, denoted by $W$ and $W^\dyn$, respectively. That is, $W$ acts on $z\in \catA$  and $W^\dyn$ acts on $\la\in \catA^\vee$. $W^\dyn$ is usually referred as the dynamical action, which is also called the K\"ahler action. 
Let   $\calQ_{W^{\dyn}}$ be the twisted product of $\cQ$ with $W^\dyn$, and $\calQ_{W\times W^{\dyn}}$ be the twisted product of  $\calQ_{W^{\dyn}}$ and $W$.
Then $\cQ_{W^{\dyn}}$ is a free left $\calQ$-module with basis $\{\de_w^{\dyn}|w\in W\}$, and $\cQ_{W\times W^{\dyn}}$ is a free left  $\cQ_{W^{\dyn}}$-module with basis $\{\de_w, w\in W\}$.  The product in $\cQ_{W^{\dyn}\times W}$  can be written as 
\begin{equation}\label{eq:twistedproduct}
\de_w\de_v^{\dyn}=\de_v^{\dyn}\de_w, \quad (a\de_w^{\dyn}\de_{v})( a'\de_{w'}^{\dyn}\de_{v'})=a\cdot{}^{vw^{\dyn}}a'\de_{ww'}^{\dyn}\de_{vv'}, \quad w,v,w', v'\in W, a,a'\in \cQ.
\end{equation}
Denote $\de_\al=\de_{s_\al}$.

We fix some notations. Let $\Sigma_w=\{\al>0|w^{-1}\al<0\}$ be the left inversion set of $w$. Denote 
\[
\theta_\Pi(\hbar\pm z)=\prod_{\al>0}\theta(\hbar\pm z_\al), \quad \theta_\Pi(\hbar\pm \la)=\prod_{\al>0}\theta(\hbar\pm\la_{\al^\vee}), \quad \theta_\Pi(z)=\prod_{\al>0}\theta(z_\al), \quad \theta_\Pi(\la)=\prod_{\al>0}\theta(\la_{\al^\vee}). 
\]
Denote $\epsilon_w=(-1)^{\ell(w)}$. Note that 
\[{}^w\theta_\Pi(z)=\epsilon_w\theta_\Pi(z), ~{}^{w^{\dyn}}\theta_\Pi(\la)=\epsilon_w\theta_\Pi(\la).\]
 Let $w_0\in W$ be the longest element.   Denote \begin{align*}\bfg=\prod_{\al>0}\frac{\theta(\hbar-z_\al)}{\theta(z_\al)}=\frac{\theta_\Pi(\hbar-z)}{\theta_\Pi(z)},&& ~\bfh=\prod_{\al>0}\frac{\theta(\hbar-\la_{\al^\vee})}{\theta(\la_{\al^\vee})}=\frac{\theta_\Pi(\hbar-\la)}{\theta_\Pi(\la)}.
\end{align*}
It is easy to see that 
\begin{align*}
{}^{w_0}\bfg=\ep_{w_0}\frac{\theta_\Pi(\hbar+z)}{\theta_\Pi(z)}, && {}^{w_0^\dyn}\bfh=\ep_{w_0}\frac{\theta_\Pi(\hbar+\la)}{\theta_\Pi(\la)}. 
\end{align*}

\subsection{The elliptic DL operators }
Let $z\in \catA, \la\in \catA^\vee$. For any simple root $\al$, define the elliptic Demazure-Lusztig (DL) operator
\[
T_{\al}^{ \la}=\frac{\theta(\hbar-z_\al)\theta(\la_{\al^\vee})}{\theta(z_\al)\theta(\hbar-\la_{\al^\vee})}\de_\al \de_\al^{\dyn}+\frac{\theta(\hbar)\theta(z_\al-\la_{\al^\vee})}{\theta(z_\al)\theta(\hbar-\la_{\al^\vee})}\de_\al^{\dyn}\in \calQ_{W^{\dyn}\times W}.
\]
Since $(-\la)_{\al^\vee}=-\la_{\al^\vee}$, so 
\begin{equation}
T_{\al}^{ -\la}=-\frac{\theta(\hbar-z_\alpha)\theta(\la_{\al^\vee})}{\theta(z_\al)\theta(\hbar+\la_{\al^\vee})}\de_\al\de_\al^{\dyn}+\frac{\theta(\hbar)\theta(z_\al+\la_{\al^\vee})}{\theta(z_\al)\theta(\hbar+\la_{\al^\vee})}\de_\al^{\dyn}, 
\end{equation}
and it also 
satisfies that $T_{\al}^{ \la}=\de_\al^{\dyn}T_{\al}^{ -\la}\de_\al^{\dyn}$. 
Furthermore, there is  an anti-involution that switches the two elements as well. See \eqref{eq:invT} below. 
It is easy to get that 
\[
\de_\al=-\frac{\theta(z_\al)\theta(\hbar+\la_{\al^\vee})}{\theta(\hbar-z_\al)\theta(\la_{\al^\vee})}\de_\al^\dyn T_\al^\la+\frac{\theta(\hbar)\theta(z_\al+\la_{\al^\vee})}{\theta(\hbar-z_\al)\theta(\la_{\al^\vee})}. 
\]
For simplicity, we will denote 
\[
\de_i=\de_{\al_i}, ~\de_i^\dyn=\de_{\al_i}^\dyn, ~ T_i^{\pm \la}=T_{\al_i}^{\pm \la}. 
\]
For the remaining calculation, we sometimes write 
\begin{equation}
\label{eq:T}\de_i^\dyn T_i^\la=p_i\de_i+q_i, ~p_i, q_i\in \cQ. 
\end{equation}

\begin{remark}\label{rem:T}
Let $\calL_\al=T\times_B\bbC_{-\al}$ be the line bundle over $G/B$. If one identifies $c_1^{coh}(\calL_\al)=z_\al, -\ln h=\hbar, \ln h^{\al^\vee}=\la_{\al^\vee}$, then $-T_{\al}^{\la}$ corresponds to the operator in \cite[Theorem 1.3]{RW19}. The notion $T_\al^{\pm\la}$ was denoted by $T_\al^{z,\pm\la}$ in \cite{LZZ23}. 
\end{remark}

One can show that 
\begin{enumerate}
\item[(i)] $(T_{\al}^{\la})^2=1$. This can be proved by direct computations by using Fay's Trisecant Identity.
\item[(ii)] The braid relations are satisfied. This is proved in \cite{RW19} by using geometric arguments. 
Define $T_{w}^\la=T_{i_1}^\la\cdots T_{i_k}^{\la}$ if $w=s_{i_1}\cdots s_{i_k}$ is a reduced decomposition. Because the braid relations are satisfied, the definition of $T^{\la}_w$ does not depend on the choice of reduced decompositions.
\end{enumerate}

We develop some basic properties of the elliptic DL operators. 
\begin{lemma}\label{lem:1}For any $a\in \cQ$, we have the Bernstein relation 
\[T_{\al}^\la a={}^{s_\al s_\al^{\dyn}}aT_{\al}^\la+\frac{\theta(\hbar)\theta(z_\al-\la_{\al^\vee})}{\theta(z_\al)\theta(\hbar-\la_{\al^\vee})}\cdot {}^{s_\al^{\dyn}}\left(a-{}^{s_\al}a\right)\de_\al^{\dyn}. \]
\end{lemma}
\begin{proof}
Following the notation in \eqref{eq:T}, we have 
\begin{align*}
T_\al^\la a-{}^{s_\al s_\al^\dyn}aT_\al^\la&=({}^{s_\al^\dyn}p_\al\de_\al+{}^{s_\al^\dyn}q_\al)\de_\al^\dyn a-{}^{s_\al s_\al^\dyn}a({}^{s_\al^\dyn}p_\al\de_\al+{}^{s_\al^\dyn} q_\al)\de_\al^\dyn\\
&={}^{s_\al^\dyn}p_\al \cdot{}^{s_\al s_\al^\dyn}a\de_\al\de_\al^\dyn+{}^{s_\al^\dyn}q_\al\cdot {}^{s_\al^\dyn}a\de_\al^\dyn-{}^{s_\al s_\al^\dyn}a\cdot {}^{s_\al^\dyn}p_\al\de_\al\de_\al^\dyn-{}^{s_\al s_\al^\dyn}a\cdot {}^{s_\al^\dyn}q_\al \de_\al^\dyn\\
&={}^{s_\al^\dyn}q_\al \cdot {}^{s_\al^\dyn}(a-{}^{s_\al}a)\de_\al^\dyn\\
&=\frac{\theta(\hbar)\theta(z_\al-\la_{\al^\vee})}{\theta(z_\al)\theta(\hbar-\la_{\al^\vee})}\cdot {}^{s_\al^\dyn}(a-{}^{s_\al}a)\de_\al^\dyn. 
\end{align*}
\end{proof}
\begin{lemma}\label{lem:2}
For any two simple roots $\al_i, \al_j$, we have 
\[
\de_i^{\dyn}T_{i}^\la\de_j^{\dyn}T_j^\la=\frac{p_i}{{}^{s_j^\dyn}p_i}\de_{ji}^{\dyn}T_{ij}^\la+(q_i-\frac{p_i\cdot{}^{s_j^\dyn}q_i}{{}^{s_j^\dyn}p_i})\de_j^{\dyn}T_j^\la. 
\]
Here $\de_{ij}^\dyn=\de_{s_is_j}^\dyn$ and $T_{ij}^\la=T_{s_is_j}^\la$. 
\end{lemma}
\begin{proof}
Following the notation in \eqref{eq:T}, we have
\begin{align*}
\frac{p_i}{{}^{s_j^\dyn}p_i}\de_{ji}^\dyn T^\la_{ij}&=\frac{p_i}{{}^{s_j^\dyn}p_i}\de_{ji}^\dyn \de_i^\dyn(p_i\de_i+q_i)\de_j^\dyn(p_j\de_j+q_j)\\
&=\frac{p_i}{{}^{s_j^\dyn}p_i}({}^{s_j^\dyn}p_i\de_i+{}^{s_j^\dyn}q_i)(p_j\de_j+q_j)\\
&=\frac{p_i}{{}^{s_j^\dyn}p_i}({}^{s_j^\dyn}p_i\cdot {}^{s_i}p_j\de_{ij}+{}^{s_j^\dyn}q_i p_j\de_j+{}^{s_j^\dyn}p_i\cdot {}^{s_i}q_j\de_i+{}^{s_j^\dyn}q_iq_j)\\
&=p_i\cdot {}^{s_i}p_j\de_{ij}+\frac{p_i \cdot{}^{s_j^\dyn}q_i p_j}{{}^{s_j^\dyn}p_i}\de_j+p_i\cdot{}^{s_i}q_j\de_i+\frac{p_i\cdot{}^{s_j^\dyn}q_iq_j}{{}^{s_j^\dyn}p_i}\\
\de_i^\dyn T_i^\la\de_j^\dyn T_j^\la&= (p_i\de_i+q_i)(p_j\de_j+q_j)\\
&=p_i\cdot {}^{s_i}p_j\de_{ij}+q_ip_j\de_j+p_i\cdot{}^{s_i}q_j\de_i+q_iq_j\\
&=\frac{p_i}{{}^{s_j^\dyn}p_i}\de_{ji}^\dyn T_{ij}^\la+(q_i-\frac{p_i\cdot{}^{s_j^\dyn}q_i}{{}^{s_j^\dyn}p_i})(p_j\de_j+q_j)\\
&=\frac{p_i}{{}^{s_j^\dyn}p_i}\de_{ji}^\dyn T_{ij}^\la+(q_i-\frac{p_i\cdot{}^{s_j^\dyn}q_i}{{}^{s_j^\dyn}p_i})\de_j^\dyn T_j^\la.
\end{align*}
The lemma is proved. 
\end{proof}
\begin{theorem}\label{thm:mat}
\begin{enumerate}
\item[(a)] We have the following transition matrices:
\begin{equation}\label{eq:1}
T^{\la}_w=\sum_{v\le w}a^{ \la}_{w,v}\de_w^{\dyn}\de_v, \quad \de_w=\sum_{v\le w}b_{w,v}^{ \la}\de^{\dyn}_{v^{-1}}T_v^{ \la}, \quad a_{w,v}^\la, b_{w,v}^\la\in \cQ,
\end{equation}
where the coefficients $a^\la_{w,v}, b^{\la}_{w,v}$ satisfy 
\begin{equation}\label{eq:invmat}\sum_{v}a^{ \la}_{w,v}\cdot{}^{w^{\dyn}}b^{ \la}_{v,u}=\de^{\Kr}_{w,u}, \quad \sum_{v}b^{ \la}_{w,v}\cdot {}^{(v^{-1})^{\dyn}}a^{ \la}_{v,u}=\de^{\Kr}_{w,u}. 
\end{equation}
Here $\de^{\Kr}_{w,u}$ is the Kronecker symbol. Moreover, 
\begin{equation}\label{eq:3}
a^{\la}_{w,w}=\prod_{\al\in \Sigma_w}\frac{\theta(\hbar-z_\al)\theta(\la_{\al^\vee})}{\theta(z_\al)\theta(\hbar-\la_{\al^\vee})}, \quad b^\la_{w,w}=\prod_{\al\in \Sigma_w}\frac{\theta(z_\al)}{\theta(\hbar-z_\al)}(w^{-1})^{\dyn}(\frac{\theta(\hbar-\la_{\al^\vee})}{\theta(\la_{\al^\vee})}).
\end{equation}
\item[(b)] The set $\{T_w^\la~| ~w\in W\}$ is a basis of the left $\cQ_{W^{\dyn}}$-module $\cQ_{W^{\dyn}\times W}$. 
\end{enumerate}

\end{theorem}
\begin{proof}
From 
\[
T_\al^\la={}^{s_\al^\dyn}p_\al \de_\al\de_\al^\dyn+{}^{s_\al^\dyn}q_\al\de_\al^\dyn,
\]
if $w=s_{i_1}s_{i_1}\cdots s_{i_k}$, we have 
\[
T_w^\la=T_{i_1}^\la\cdots T_{i_k}^\la=({}^{s_{i_1}^\dyn}p_{i_1} \de_{i_1}\de_{i_1}^\dyn+{}^{s_{i_1}^\dyn}q_{i_1}\de_{i_1}^\dyn)\cdots ({}^{s_{i_k}^\dyn}p_{i_k} \de_{i_k}\de_{i_k}^\dyn+{}^{s_{i_k}^\dyn}q_{i_k}\de_{i_k}^\dyn).
\] 
Opening the brackets, it is standard to obtain the formula in the first part of \eqref{eq:1}.  Moreover, the leading coefficient $a^\la_{w,w}$ is given  by 
\[
\prod_{\al\in \Sigma_w}\frac{\theta(\hbar-z_\al)\theta(\la_{\al^\vee})}{\theta(z_\al)\theta(\hbar-\la_{\al^\vee})}.
\]
This gives the first identity of \eqref{eq:3}. 

We consider the second part of \eqref{eq:1}, that is, we express $\de_v$ as a linear combination of $T_w^\la$.  Write
\[
\de_i=a_i+b_i\de_i^\dyn T^\la_i,~a_i, b_i\in \cQ. 
\]
From Lemmas \ref{lem:1} and \ref{lem:2}, we can write
\[ T_i^\la a={}^{s_is_i^\dyn }aT_i^\la+f_i(a)\de_i^\dyn, \quad \de_i^\dyn T_i^\la \de_j ^\dyn T_j^\la=f_{ij}\de_{ji}^\dyn T_{ji}^\la+g_{ij}\de_j^\dyn T_j^\la, ~f_i(a), f_{ij}, g_{ij}\in \cQ. 
\]
We have
\begin{align*}
\de_{ij}&=\de_i\de_j\\
&=(a_i+b_i\de_i^\dyn T_i^\la)(a_j+b_j\de_j^\dyn T_j^\la)\\
&=a_ia_j+a_i b_j\de_j^\dyn T_j^\la+b_i\de_i^\dyn T_i^\la a_j+b_i \de_i^\dyn T_i^\la b_j\de_j^\dyn T_j^\la\\
&=a_ia_j+a_ib_j\de_j^\dyn T_j^\la+b_i({}^{s_i}a_j\de_i^\dyn T_i^\la+{}^{s_i^\dyn}f_i(a_j))+b_i({}^{s_i}b_j\de_i^\dyn T_i^\la+{}^{s_i^\dyn}f_i(b_j)) \de_j^\dyn T_j^\la\\
&=a_ia_j+a_ib_j\de_j^\dyn T_j^\la+b_i({}^{s_i}a_j\de_i^\dyn T_i^\la+{}^{s_i^\dyn}f_i(a_j))+b_i\cdot {}^{s_i}b_j \de_i^\dyn T_i^\la\de_j^\dyn T_j^\la+b_i\cdot {}^{s_i^\dyn }f_i(b_j)\de_j^\dyn T_j^\la.
\end{align*}
By induction, the second expression in \eqref{eq:1} follows.

 Since the matrix $(a^\la_{w,v}\de_{w}^{\dyn})_{(w,v)\in W^2}$ is upper triangular with diagonal entries all invertible in $\cQ_{W^\dyn}$, so it is  invertible in $\cQ_{W^\dyn}$,  and part (b) then follows. 

We now consider \eqref{eq:invmat}. 
We have
\begin{align*}
T_w^\la=\sum_{v}a^\la_{w,v}\de_w^\dyn\de_v=\sum_va^\la_{w,v}\de_w^\dyn \sum_u b^\la_{v,u}\de_{u^{-1}}^\dyn T_u^\la=\sum_{u}(\sum_va^\la_{w,v}\cdot{}^{w^\dyn}b^\la_{v,u}\de_{wu^{-1}}^{\dyn})T_u^\la.
\end{align*} 
so 
\[
\de_{w,u}^{\Kr}=\sum_v a^\la_{w,v}\cdot {}^{w^\dyn}b^\la_{v,u}. 
\]
The second identity in \eqref{eq:invmat} can be proved easily. 

Lastly, we consider the second part of \eqref{eq:3}. From \eqref{eq:invmat}, we see that 
\[
b^\la_{w,w}\cdot {}^{(w^{-1})^\dyn}a^\la_{w,w}=1,
\]
so the expression for $b^\la_{w,w}$ follows. This concludes the proof of the theorem.
\end{proof}

As a special case, we have
\begin{align}\label{eq:abw0}
a^\la_{w_0, w_0}&=\frac{\theta_\Pi(\hbar-z)\theta_\Pi(\la)}{\theta_\Pi(z)\theta_\Pi(\hbar-\la)}=\frac{\bfg}{\bfh}, & b^\la_{w_0, w_0}&=\ep_{w_0}\frac{\theta_\Pi(z)\theta_\Pi(\hbar+\la)}{\theta_\Pi(\hbar-z)\theta_\Pi(\la)}=\frac{{}^{w_0^\dyn}\bfh}{\bfg},\\
 ~a^{-\la}_{w_0,w_0}&=\frac{\bfg}{{}^{w_0^\dyn}\bfh},& ~b^{-\la}_{w_0, w_0}&=\frac{\bfh}{\bfg}.
\end{align}
 
For the root system of type $A_2$, the `$a$'-coefficients are computed in \cite[Example 8.3]{LZZ23}.

\subsection{The anti-involution}
It turns out that one can relate $T_\al^{\la}$ and $T_\al^{-\la}$ by using an anti-involution as follows. We will also show that these two operators are adjoint to each other.
We define an anti-involution 
$\iota:\cQ_{W^{\dyn}\times W}\to \cQ_{W^{\dyn}\times W}$, 
\[ a\de_v^{\dyn}\de_w\mapsto \de_{v^{-1}}^{\dyn}\de_{w^{-1}} a\frac{{}^w\bfg}{\bfg}={}^{(v^{-1})^\dyn w^{-1}}a\cdot \frac{\bfg}{{}^{w^{-1}}\bfg}\de_{v^{-1}}^\dyn\de_{w^{-1}}.\]
This anti-involution was defined in \cite{LZZ23}. 
For any simple root $\alpha$, it is easy to see that 
\[
{}^{s_\al}\bfg=\bfg\cdot (-1)\cdot\frac{\theta(\hbar+z_\al)}{\theta(\hbar-z_\al)}, 
\]
so one can verify 
\begin{equation}\label{eq:invT}
\iota(T^\la_i)=T^{-\la}_i, \quad \iota(T^\la_w)=T^{-\la}_{w^{-1}}.
\end{equation}

The following lemma deals with the coefficients after applying the anti-involution. 
\begin{lemma}\label{lem:reverse}For any $w,v\in W$, we have 
\[
{}^v(\frac{a^\la_{w^{-1}, v^{-1}}}{\bfg})=\frac{{}^{(w^{-1})^\dyn}a^{-\la}_{w,v}}{\bfg}. 
\]
\end{lemma}

\begin{proof}We have
\begin{align*}
\sum_{v}a_{w^{-1}, v^{-1}}^{ \la}\de_{w^{-1}}^{\dyn}\de_{v^{-1}}&=T_{w^{-1}}^{ \la}\\
&=\iota(T_w^{-\la})\\
&=\iota(\sum_va_{w,v}^{-\la}\de_w^{\dyn}\de_v)\\
&=\sum_v{}^{v^{-1}(w^{-1})^{\dyn}}a_{w,v}^{-\la}\cdot \frac{\bfg}{{}^{v^{-1}}\bfg}\de_{w^{-1}}^{\dyn}\de_{v^{-1}}.
\end{align*}
The formula then follows from the comparison of the coefficients of $\de_{w^{-1}}^{\dyn}\de_{v^{-1}}$. 
\end{proof}
\begin{example}
If $w=v=w_0$, then 
\[
{}^w(\frac{a^\la_{w^{-1}, w^{-1}}}{\bfg})={}^w(\frac{\bfg}{\bfh}\cdot \frac{1}{\bfg})=\frac{1}{\bfh}, ~
\frac{{}^{(w^{-1})^{\dyn}}a^{- \la}_{w,w}}{\bfg}={}^{(w_{0}^{-1})^\dyn}\left(\frac{\bfg}{{}^{w_0^\dyn}\bfh}\right)\frac{1}{\bfg}=\frac{1}{\bfh},
\]
so Lemma \ref{lem:reverse} is verified in this case. 
\end{example}

\section{The dual as a left  $\cQ_{W^{\dyn}}$-module}
In this section, we define a dual of the twisted group algebra $\cQ_{W^\dyn\times W}$. This serves as an intermediate step in the development of the algebraic model for the elliptic Schubert classes.

\subsection{}
We define the following dual:
\begin{equation}\label{eq:dualdef}
\cQ_{W^{\dyn}\times W}^\star=\Hom_{\cQ_{W^{\dyn}}}(\cQ_{W^{\dyn}\times W}, \cQ_{W^{\dyn}}),
\end{equation}
where $\cQ_{W^{\dyn}\times W}$ is considered as a free left $\cQ_{W^{\dyn}}$-module, with two bases $T_{w}^{ \la}$ and $\de_w$. Then $\cQ_{W^{\dyn}\times W}^\star$ becomes a right $\cQ_{W^{\dyn}}$-module  coming  from the right multiplication of $\cQ_{W^{\dyn}}$ on itself. That is, 
\[
(gp)(x)=g(x)p, \quad g\in \cQ_{W^\dyn\times W}^\star, x\in \cQ_{W^\dyn\times W}, p\in \cQ_{W^\dyn}. 
\]
So $\cQ_{W^{\dyn}\times W}^\star$ is a free right $\cQ_{W^{\dyn}}$-module with basis $g_w$ dual to the basis $\de_w \in \cQ_{W^{\dyn}\times W}$. We also denote by $(T_{w}^{ \la})^\star$  the dual basis of $T_{w}^{ \la}$. Following from  Theorem \ref{thm:mat}, we have 
\begin{align}\label{eq:dual}
(T_w^{ \la})^\star&=\sum_vg_v [(T^\la_w)^\star(\de_v)]=\sum_vg_v [(T^\la_w)^\star(\sum_ub^\la_{v,u}\de_{u^{-1}}^\dyn T_u^\la)]=\sum_{v}g_{v}b^{ \la}_{v,w}\de_{w^{-1}}^{\dyn}.
\end{align}

The right $\cQ_{W^{\dyn}}$-module  $\cQ_{W^{\dyn}\times W}^\star$ is also a  left $\cQ_{W^{\dyn}\times W}$-module induced by the right multiplication of $\cQ_{W^{\dyn}\times W}$ on itself. More precisely, we have
\[
(x\bullet g)(x')=g(x'x), \quad x,x'\in \cQ_{W^{\dyn}\times W},~ g\in \cQ_{W^{\dyn}\times W}^\star.
\]
\begin{lemma}The $\bullet$-action is right $\cQ_{W^{\dyn}}$-linear.
\end{lemma}
\begin{proof}This follows from the fact the $\bullet$-action comes from the domain in \eqref{eq:dualdef}, and the right $\cQ_{W^{\dyn}}$-action comes from the codomain of \eqref{eq:dualdef}. One can also check directly:
for any $x,x'\in \cQ_{W^{\dyn}\times W}, p\in \cQ_{W^{\dyn}}, g\in \cQ_{W^{\dyn}\times W}^\star$, we have
\[
(x\bullet (gp))(x')=(gp)(x'x)=g(x'x)p=(x\bullet g)(x')p=((x\bullet g)p)(x'),
\]
so $(x\bullet (gp))=(x\bullet g)p$. 
\end{proof}
\begin{lemma}
We have the following explicit formulas of the $\bullet$-action:
\[
\de_w\bullet g_vb\de_u^{\dyn}=g_{vw^{-1}}b\de_u^{\dyn}, \quad a\de_w^{\dyn}\bullet g_vb\de_u^{\dyn}=g_v\cdot{}^va\cdot{}^{w^{\dyn}}b\de_{wu}^{\dyn}, \quad a, b\in \cQ. 
\]
\end{lemma}
\begin{proof}
Since the $\bullet$-action is right $\cQ_{W^\dyn}$-linear, it suffices  to evaluate the elements on $\de_x, x\in W$. 
We have 
\begin{align*}
(\de_w\bullet g_vb\de_u^\dyn)(\de_x)&=g_v(\de_x\de_w)b\de_u^\dyn=\de^{\Kr}_{xw,v}b\de_u^\dyn,\\
(a\de_w^\dyn\bullet g_vb\de_u^\dyn)(\de_x)&=g_v(\de_xa\de_w^\dyn)b\de_u^\dyn=\de_{x,v}^{\Kr}\cdot {}^xa\cdot {}^{w^\dyn}b\de^\dyn_{wu},
\end{align*}
so the identities follow. 
\end{proof}

The bimodule $\cQ_{W^{\dyn}\times W}^\star$ has a non-commutative product structure, defined as follows:
\[
(g_wp)\cdot (g_vq)=g_v\de^{\Kr}_{w,v}pq, \quad p,q\in \cQ_{W^{\dyn}}. 
\]
The multiplicative identity is $\unit=\sum_{w\in W}g_w$.  Concerning the compatibility of the product with the $\bullet$-action, we have the following special cases: 
\begin{lemma}For any $p\in \cQ_{W^{\dyn}}, w\in W, g,g'\in \cQ_{W^{\dyn}\times W}^\star$, we  have
\[
p\bullet (g\cdot g')=(p\bullet g) \cdot g',\quad 
\de_w\bullet (g\cdot g')=(\de_w\bullet g)\cdot (\de_w\bullet g').
\]
\end{lemma}
Note that in general $p\bullet (g\cdot g')\neq g\cdot (p\bullet g')$.
\begin{proof}
Let $p=b\de_{w_1}^{\dyn}, g=g_va\de_{u}^{\dyn}, g'=g_{v'}a'\de_{u'}^{\dyn}$, then 
\begin{align*}
p\bullet (g\cdot g')=b\de_{w_1}^{\dyn}\bullet (g_va\de_{u}^{\dyn}\cdot g_{v'}a'\de_{u'}^{\dyn})=b\de_{w_1}^{\dyn}\bullet (g_v\de_{v,v'}^{\Kr}a\de_u^{\dyn}a'\de_{u'}^{\dyn})=g_v\de_{v,v'}^{\Kr}{}^vb
\cdot \de_{w_1}^{\dyn}a\de_u^{\dyn}a'\de_{u'}^{\dyn}=(p\bullet g)\cdot g'.
\end{align*}
Similarly, 
\[
\de_w\bullet (g\cdot g')=\de_w\bullet ((g_v\de_{v,v'}^{\Kr}a\de_u^{\dyn}a'\de_{u'}^{\dyn}))=g_{vw^{-1}}\de_{v,v'}^{\Kr}a\de_u^{\dyn}a'\de_{u'}^{\dyn})=(\de_w\bullet g)\cdot (\de_w\bullet g'). 
\]
\end{proof}

\subsection{}
We will be considering the class $T^{-\la}_{w^{-1}w_0}\bullet g_{w_0}$, which is related with the opposite Schubert classes defined below. It coincides with the dual basis $(T_w^{-\la})^\star$ up to a normalization.
\begin{lemma} \label{lem:Twithdual} In $\cQ_{W^{\dyn}\times W}^\star$, we have 
\begin{align*}
T^{\la}_{w^{-1}w_0}\bullet g_{w_0}=(T^{\la}_w)^\star\frac{\bfg}{\bfh}\de_{w_0}^{\dyn}, &&T^{-\la}_{w^{-1}w_0}\bullet g_{w_0}=(T^{-\la}_w)^\star\frac{\bfg}{{}^{w_0^\dyn}\bfh}\de_{w_0}^{\dyn}.
\end{align*}
\end{lemma}
\begin{proof}By definition of the $\bullet$-action, we have
\begin{align*}
(T^{\la}_{w^{-1}w_0}\bullet g_{w_0} )(T^{\la}_{u})&=g_{w_0}(T^{\la}_{u}T^{\la}_{w^{-1}w_0})\\
&=g_{w_0}(T^{\la}_{uw^{-1}w_0})\\
&\overset{\text{Thm. } \ref{thm:mat}}=g_{w_0}(\sum_{v\le uw^{-1}w_0 }a^{\la}_{uw^{-1}w_0, v}\de_{uw^{-1}w_0}^{\dyn}\de_v)\\
&=\de_{w_0, uw^{-1}w_0}^{\Kr}a^{\la}_{w_0, w_0}\de_{w_0}^{\dyn}\\
&\overset{\eqref{eq:abw0}}=\de_{u,w}^{\Kr}\frac{\bfg}{\bfh}\de_{w_0}^{\dyn}=((T^\la_w)^\star\frac{\bfg}{\bfh}\de_{w_0}^{\dyn} )(\de_u). 
\end{align*}
Therefore, the  first identity is proved. The second one can be obtained by replacing $\la$ by $-\la$. 
\end{proof}
Below we will use the $-\la$-version of this lemma.  Following from \eqref{eq:dual} and Lemma  \ref{lem:Twithdual}, we have
\begin{align}
T^{-\la}_{w^{-1}w_0}\bullet g_{w_0}&\overset{\text{Lem. }\ref{lem:Twithdual}}=(T^{-\la}_w)^\star\frac{\bfg}{{}^{w_0^\dyn}\bfh}\de_{w_0}^\dyn\\
&\overset{\eqref{eq:dual}}=\sum_vg_vb^{-\la}_{v,w}\de_{w^{-1}}^\dyn\frac{\bfg}{{}^{w_0^\dyn}\bfh}\de_{w_0}^\dyn\\
\label{eq:Tdualres}&=\sum_{v\ge w}g_vb_{v,w}^{-\la}\cdot \frac{\bfg}{{}^{(w^{-1}w_0)^{\dyn}}\bfh}\de_{w^{-1}w_0}^{\dyn}. 
\end{align}

\begin{example}
If $w=w_0$, then $T^{-\la}_{w^{-1}w_0}\bullet g_{w_0}=g_{w_0}$, and the right hand side of the identity in \eqref{eq:Tdualres} is 
\[
g_{w_0}b^{-\la}_{w_0, w_0}\frac{\bfg}{\bfh}\overset{\eqref{eq:abw0}}=g_{w_0}\frac{\bfh}{\bfg}\frac{\bfg}{\bfh}=g_{w_0}.
\]
\end{example}

\section{The dual as a  $\cQ$-module} \label{sec:dual}
In this section, we define the elliptic Schubert classes and the opposite elliptic Schubert classes, and show that they are dual to each other.

\subsection{The $\cQ$-dual} 
Define\[
\cQ_W^*=\Hom(W,\cQ),
\]
which is a free right $\cQ$-module with basis $f_w$ such that $f_w(v)=\de_{w,v}$.    Indeed, because $\cQ$ is commutative, it does not matter whether we write the coefficients on the left or on the right.  Note that we use $()^\star$ for the dual of left $\cQ_{W^\dyn}$-module and $()^*$ for the dual of left $\cQ$-module. 
The module $\cQ^*_W$ has a commutative product coming from that of $\cQ$. That is, 
\[
(f_wa)\cdot (f_vb)=\de_{w,v}^{\Kr}f_wab, \quad a,b\in \cQ. 
\]
The multiplicative identity is also denoted by $\unit=\sum_wf_w. $

Similar as above, there is also an action of $\cQ_{W^{\dyn}\times W}$ on $\cQ_W^*$:
\begin{equation}\label{eq:bullet}
a\de_w^{\dyn}\de_v\bullet (f_ub)=f_{uv^{-1}}\cdot{}^{uv^{-1}}a \cdot {}^{w^{\dyn}}b, \quad a,b\in \cQ. 
\end{equation}
The action of $\de_w^{\dyn}\in \cQ_{W^{\dyn}\times W}$ is not $\cQ$-linear, but the action of $a\de_v\in\cQ_W\subset\cQ_{W^{\dyn}\times W}$ is. Note that it is different from that in \cite{LZZ23}, but it is compatible with the action considered in \cite{RW19}.  Other than the dynamical Weyl group action, this $\bullet$-action behaves similarly as the corresponding actions in general oriented cohomology theories (cohomology and K-theory) considered in Kostant-Kumar and Calm\`es-Zainoulline-Zhong \cite{KK86, KK90, CZZ15}. 

Indeed, one can view $\cQ_W^*$ as the commutative subring inside $\cQ_{W^{\dyn}\times W}^\star$ induced by $\cQ\to \cQ_{W^{\dyn}}, a\mapsto a\de_e^{\dyn}$. Or equivalently, we have an  embedding of rings
\[
\phi:\cQ^*_W\to \cQ^\star_{W^{\dyn}\times W}, \quad f_wa\mapsto g_wa\de_e^{\dyn}, \quad a\in \cQ. 
\] If we view the right $\cQ_{W^\dyn}$-module $\cQ^\star_{W^{\dyn}\times W}$ as a right $\cQ$-module via the embedding $\cQ\to \cQ_{W^{\dyn}}$, then $\phi$ is right $\cQ$-linear.

We also have a surjection of abelian groups
\[\pi:\cQ_{W^{\dyn}\times W}^\star\to \cQ_W^*, \quad g_wa\de_{v}^{\dyn}\mapsto f_wa, \quad a\in \cQ.\]
This map does not preserve products, and it is not right $\cQ$-linear. However,  
\begin{lemma}  \label{lem:equiv}The map $\pi$ is $\cQ_{W^{\dyn}\times W}$-equivariant via the `$\bullet$'-action, that is, $\pi (x\bullet g)=x\bullet \pi(g)$ for any $x\in \cQ_{W^{\dyn}\times W}, g\in \cQ_{W^{\dyn}\times W}^\star$. 
\end{lemma}

\begin{proof}
Let $x=a\de_w^{\dyn}\de_v\in \cQ_{W^{\dyn}\times W}, g=g_ub\de_{u'}^{\dyn}\in \cQ^\star_{W^{\dyn}\times W}$, then
\begin{align*}
\pi(a\de_w^{\dyn}\de_v\bullet g_ub\de_{u'}^{\dyn})&=\pi(g_{uv^{-1}}\cdot{}^{uv^{-1}}a\cdot {}^{w^{\dyn}}b\de_{wu'}^{\dyn})=f_{uv^{-1}}\cdot {}^{uv^{-1}}a\cdot {}^{w^{\dyn}}b=a\de_w^{\dyn}\de_v\bullet f_ub=a\de_w^{\dyn}\de_v\bullet \pi(g). 
\end{align*}
\end{proof}
Starting from now, we only work with $\cQ_W^*$. 

\subsection{Elliptic Schubert classes}
We  define the main objects of this paper. 
\begin{definition}Let $\la\in \catA^\vee$. We define the elliptic Schubert class associated to $\la$ to be $\bfE^\la_w=T^\la_{w^{-1}}\bullet f_e\in \cQ_W^*$, and the opposite elliptic Schubert class to be $\Ell^{-\la}_w=T^{-\la}_{w^{-1}w_0}\bullet f_{w_0}\in \cQ_W^*$.
\end{definition}
\begin{remark} Recall that our operator $T^\la_\al$ agrees with the operator $T^{z,\la}_\la$ in \cite{LZZ23}, but the action $\bullet$ is different from that in \cite{LZZ23} (see \S~\ref{sec:dual}). Moreover, our elliptic class $\bfE_e^\la=f_e$ where the elliptic class $\bfE^{z,\la}_e$ in \cite[Appendix A]{LZZ23} is equal  to $\mathfrak{c}f_e$ where $\mathfrak{c}$ is a rational section of the twisted Poincar\'e line bundle $\mathbb{L}$. 

On the other hand, by using the identification spelled out in Remark \ref{rem:T}, one can identify the restriction $\bfE_w^\la|_\sigma$ with the classes $E_\sigma(X_w,\la)$ in \cite{RW19} with $\sigma\in W. $
\end{remark}
The elliptic Schubert classes $\bfE^\la_w$ were defined in \cite{RW19}. We will show below that they are dual to the opposite Schubert classes $\Ell^{-\la}_w$. 

By definition and the relation that $T^\la_wT_v^\la=T_{wv}^\la$,  we have the following recursive formulas:
\[
T_i^\la\bullet \bfE^\la_w=\bfE^\la_{ws_i}, \quad T_i^{-\la}\bullet \Ell^{-\la}_{w}=\Ell^{-\la}_{ws_i}. 
\]
The first one was referred as the Bott-Samelson recursion in \cite{RW19}. 

\begin{remark}
We can define a right action of $T_i^\la$ on $\bfE^\la_w$ by 
\[
 \bfE^\la_w\odot T_i^\la:=T_{w^{-1}}^\la T_i^\la\bullet \bfE^\la_e=\bfE^\la_{s_iw}.
\]
This gives the R-matrix recursion of the elliptic Schubert classes in \cite{RW19}. 
This action permutes the elliptic Schubert classes, but they do not extend to actions of $\cQ_{W}$ or $\cQ_{W^\dyn\times W}$ on $\cQ_W^*$.
\end{remark}

We define the auxillary elliptic Schubert classes 
\[
(T_w^{-\la})^*:=\pi((T_w^{-\la})^\star)=\sum_{v\ge w}f_vb^{-\la}_{v,w}, \]
and we have
\begin{align}\label{eq:piEllopp}
\calE_w^{-\la}=T^{-\la}_{w^{-1}w_0}\bullet f_{w_0}\overset{\text{Lem. }\ref{lem:equiv}}=\pi(T^{-\la}_{w^{-1}w_0}\bullet g_{w_0})
\overset{\eqref{eq:Tdualres}}=\sum_{v\ge w}f_vb_{v,w}^{-\la}\frac{\bfg}{{}^{(w^{-1}w_0)^\dyn}\bfh}
=(T_w^{-\la})^*\frac{\bfg}{{}^{(w^{-1}w_0)^\dyn}\bfh}. 
\end{align}
From Theorem \ref{thm:mat}, we have the following restriction formula:
\begin{align}\label{eq:ellres}
\bfE_w^\la&=\sum_{v\le w}f_{v} \cdot{}^va^\la_{w^{-1},v^{-1}}.
\end{align}

The following lemma shows the relation between the `$a$'-coefficients and the `$b$'-coefficients.

\begin{lemma}We have
\[
u(a^{-\la}_{w^{-1}w_0
, u^{-1}w_0})=b^{-\la}_{u,w}\frac{\bfg}{{}^{(w^{-1}w_0)^\dyn}\bfh}. 
\]
\end{lemma}
\begin{proof}
We have 
\begin{align*}\Ell^{-\la}_w&=T^{-\la}_{w^{-1}w_0}\bullet f_{w_0}=\sum_{v\le w^{-1}w_0}a^{-\la}_{w^{-1}w_0,v}\de^{\dyn}_{w^{-1}w_0}\de_v\bullet f_{w_0}=\sum_{v\le w^{-1}w_0}f_{w_0v^{-1}}\cdot{}^{w_0v^{-1}}a^{-\la}_{w^{-1}w_0, v}. 
\end{align*}
Letting $u=w_0v^{-1}$, and comparing with \eqref{eq:piEllopp}, we obtain the result. 
\end{proof}
\subsection{The Poincar\'e Pairing}
Define
\begin{equation}
Y'_\Pi=\sum_{w\in W}\de_w\prod_{\al>0}\frac{\theta(z_\alpha)}{\theta(\hbar-z_\al)}=\sum_{w\in W}\de_w\frac{1}{\bfg}. 
\end{equation}
The $\bullet$-action of $Y_\Pi'$ provides the following duality:
\begin{lemma}\label{lem:duality1}We have 
\[
Y_\Pi'\bullet \left(\bfE_v^\la\cdot (T_w^{-\la})^*\right)=\unit\frac{\de^{\Kr}_{v,w}}{\bfg}.
\]
\end{lemma}
\begin{proof}
From \eqref{eq:ellres}, we have
\begin{align*}\bfE^\la_v=T^{ \la}_{v^{-1}}\bullet f_e=\sum_{w_1\in W}f_{w_1}\cdot{}^{w_1}a^{ \la}_{v^{-1}, w_1^{-1}}.
\end{align*}
Equation \eqref{eq:dual} gives   $(T_w^{- \la})^*=\sum_{v}f_{v}b^{- \la}_{v,w}$, so 
\begin{align*}
Y'_\Pi\bullet \left( \bfE_v^\la\cdot (T^{- \la}_w)^*\right)&=Y'_\Pi\bullet \left (\sum_{w_1\in W}f_{w_1}\cdot{}^{w_1}a_{v^{-1}, w_1^{-1}}^{ \la}\cdot b^{- \la}_{w_1, w}\right)\\
&=\sum_{w_2\in W}\de_{w_2}\frac{1}{\bfg}\bullet \left(\sum_{w_1}f_{w_1}\cdot{}^{w_1}a_{v^{-1}, w_1^{-1}}^{ \la}\cdot b^{- \la}_{w_1,w}\right)\\
&=\sum_{w_2, w_1\in W}f_{w_1w_2^{-1}}\frac{{}^{w_1}a^{ \la}_{v^{-1}, w_1^{-1}}\cdot b^{- \la}_{w_1, w}}{{}^{w_1}\bfg}\\
&\overset{w_3:=w_1w_2^{-1}}=(\sum_{w_3\in W}f_{w_3})\sum_{w_1\in W}\frac{{}^{w_1}a^{ \la}_{v^{-1}, w_1^{-1}}}{{}^{w_1}\bfg}b^{- \la}_{w_1, w}\\
&\overset{\text{Lem.} ~\ref{lem:reverse}}=\unit \sum_{w_1\in W} \frac{{}^{(v^{-1})^{\dyn}}a^{-\la}_{v,w_1}\cdot b_{w_1,w}^{-\la}}{\bfg}\\
&\overset{(*)}=\unit\frac{\de^{\Kr}_{v,w}}{\bfg}.
\end{align*}
Here equality $(*)$ follows from applying $(w^{-1})^{\dyn}$ on the second  identity of  \eqref{eq:invmat}. 
\end{proof}

To work with the opposite Schubert classes, we need to consider the following element in $\cQ_{W^{\dyn}\times W}$:
\begin{equation}
Y_\Pi=\sum_{v, w\in W}\de_w\de_ v^{\dyn}\prod_{\al>0}\frac{\theta(z_\alpha)}{\theta(\hbar-z_\al)}=\sum_{v\in W}\de_v^{\dyn}\sum_{w\in W}\de_w\frac{1}{\bfg}=(\sum_{v\in W}\de_v^{\dyn})Y_\Pi'.
\end{equation}
The map 
\[Y_\Pi\bullet\_:\cQ_W^*\to \cQ_W^*, f\mapsto Y_\Pi\bullet f\]
 is the algebraic model for the push-forward to the base point. So $\langle\_, \_\rangle:=Y_\Pi\bullet(\_\cdot \_)$ defines the Poincar\'e pairing. 

\begin{corollary}
We have
\[
Y_\Pi\bullet (\bfE^\la_w\cdot (T^{-\la}_{v})^*)=\unit\frac{\de_{w,v}^{\Kr}|W|}{\bfg}. 
\]
\end{corollary}
\begin{proof}This follows from Lemma \ref{lem:duality1} and the identity $Y_\Pi=(\sum_{v\in W}\de_v^{\dyn})Y_\Pi'$. 
\end{proof}

\begin{lemma}
We have the following adjoint property
\[
Y_\Pi\bullet  \left((x\bullet f) \cdot f'\right)=Y_\Pi\bullet\left (f \cdot (\iota(x)\bullet f')\right), \quad f,f'\in \cQ_W^*, x\in \cQ_{W^{\dyn}\times W}.
\]
\end{lemma}
\begin{proof} Note that $\bfg$ is invariant under the dynamical action.  Let $f=f_{v}a$, $f'=f_{v'}a'$, $x=b\de_{w_1}^{\dyn}\de_{w_2}$ with $a,a',b
\in \cQ$, then by definition, $\iota(x)=\de_{w_1^{-1}}^{\dyn}\de_{w_2^{-1}}b\frac{{}^{w_2}\bfg}{\bfg}$. So 
\begin{align*}
Y_\Pi\bullet ((x\bullet f)\cdot f')&=Y_\Pi\bullet \left( (b\de_{w_1}^{\dyn}\de_{w_2}\bullet f_va)\cdot f_{v'}a'\right)\\
&=Y_\Pi\bullet \left((f_{vw_2^{-1}}\cdot {}^{vw_2^{-1}}b\cdot {}^{w_1^{\dyn}}a)\cdot ( f_{v'}a')\right)\\
&=\sum_{w_3, w_4\in W}\de_{w_3}^{\dyn}\de_{w_4}\frac{1}{\bfg}\bullet \left(f_{v'}\de_{v', vw_2^{-1}}^{\Kr}{}^{vw_2^{-1}}b\cdot {}^{w_1^{\dyn}}a\cdot a'\right)\\
&=\sum_{w_3,w_4\in W}f_{v'w_4^{-1}}\frac{\de_{v', vw_2^{-1}}^{\Kr}\cdot {}^{w_3^{\dyn}vw_2^{-1}}b\cdot {}^{(w_3w_1)^{\dyn}}a\cdot {}^{w_3^{\dyn}}a'}{{}^{v'}\bfg}\\
&=(\sum_{w_4\in W}f_{vw_4^{-1}})\sum_{w_3\in W}\frac{\de_{v', vw_2^{-1}}^{\Kr}\cdot{}^{w_3^{\dyn}vw_2^{-1}}b\cdot{}^{(w_3w_1)^{\dyn}}a\cdot {}^{w_3^{\dyn}}a'}{{}^{v'}\bfg}\\
&=\unit\sum_{w_3\in W}\frac{\de_{v', vw_2^{-1}}^{\Kr}\cdot{}^{w_3^{\dyn}vw_2^{-1}}b\cdot{}^{(w_3w_1)^{\dyn}}a\cdot {}^{w_3^{\dyn}}a'}{{}^{v'}\bfg}. 
\end{align*}
On the other hand, we have
\begin{align*}
Y_\Pi\bullet (f\cdot (\iota(x)\bullet f'))&=Y_\Pi \bullet \left(f_va\cdot(\de_{w_1^{-1}}^{\dyn}\de_{w_2^{-1}}b\frac{{}^{w_2}\bfg}{\bfg}\bullet f_{v'}a')\right)\\
&=Y_\Pi\bullet  \left(f_va\cdot [f_{v'w_2}\cdot{}^{v'(w_1^{-1})^{\dyn}}b\cdot \frac{{}^{v'w_2}\bfg}{{}^{v'}\bfg}\cdot{}^{(w_1^{-1})^{\dyn}}a']\right)\\
&=\sum_{w_3, w_4\in W}\de_{w_3}^{\dyn}\de_{w_4}\frac{1}{\bfg}\bullet \left(f_v\de_{v,v'w_2}^{\Kr}\cdot a\cdot{}^{v'(w_1^{-1})^{\dyn}}b\cdot\frac{{}^{v'w_2}\bfg}{{}^{v'}\bfg}\cdot{}^{(w_1^{-1})^{\dyn}}a'\right)\\
&=\sum_{w_3,w_4\in W}f_{vw_4^{-1}}\de_{v,v'w_2}^{\Kr}\cdot\frac{1}{{}^v\bfg}\cdot {}^{w_3^{\dyn}}a\cdot{}^{v'(w_3w_1^{-1})^{\dyn}}b\cdot\frac{{}^{v'w_2}\bfg}{{}^{v'}\bfg}\cdot{}^{(w_3w_1^{-1})^{\dyn}}a'\\
&=\unit\sum_{w_3\in W}\de_{v,v'w_2}^{\Kr}\cdot\frac{1}{{}^v\bfg}\cdot{}^{w_3^{\dyn}}a\cdot{}^{v'(w_3w_1^{-1})^{\dyn}}b\cdot\frac{{}^{v'w_2}\bfg}{{}^{v'}\bfg}\cdot{}^{(w_3w_1^{-1})^{\dyn}}a'.
\end{align*}
Note that $v'=vw_2^{-1}$ if and only if $v'w_2=v$. 
If we let $w_3'=w_3w^{-1}$, then it is straightforward to see that 
\[
Y_\Pi\bullet (f\cdot (\iota(x)\bullet f'))=Y_\Pi\bullet ((x\bullet f)\cdot f'). 
\]
The conclusion is proved. 
\end{proof}

Using  \eqref{eq:invT}, we get the following corollary.
\begin{corollary}\label{cor:adj}We have
\[
Y_\Pi\bullet ((T_i^\la\bullet f)\cdot f')=Y_\Pi\bullet (f\cdot (T_i^{-\la}\bullet f')), \quad f,f'\in \cQ_{ W}^*. 
\]
\end{corollary}

The following theorem is the algebraic model of the Poincar\'e duality between the elliptic Schubert classes and opposite elliptic Schubert classes:
\begin{theorem}\label{thm:poincare}We have 
\[
Y_\Pi\bullet (\bfE^\la_w \cdot \Ell^{-\la}_v)=\unit\de_{w,v}^{\Kr}\sum_{u_1\in W}u_1^{\dyn}(\frac{\theta_\Pi(\la)}{\theta_\Pi(\hbar-\la)}). 
\]
\end{theorem}
We provide two proofs of this theorem, which are not entirely independent but have different flavors.
\begin{proof}[Proof 1]
We have
\begin{align*}
Y_\Pi\bullet (\bfE^\la_w\cdot \Ell^{-\la}_v)&\overset{\eqref{eq:piEllopp} }=\sum_{u_1\in W}\de_{u_1}^{\dyn}Y'_\Pi\bullet \left(\bfE^\la_w\cdot (T^{-\la}_v)^*\frac{\bfg}{{}^{(w^{-1}w_0)^{\dyn}}\bfh}\right) \\
&=\sum_{u_1\in W}\de_{u_1}^{\dyn}\bullet [Y'_{\Pi}\bullet \left(\bfE_w^\la\cdot (T^{-\la}_{v})^*\frac{\bfg}{{}^{(w^{-1}w_0)^\dyn}\bfh}\right)]\\
&\overset{(*)}=\sum_{u_1\in W}\de_{u_1}^{\dyn}\bullet [\unit\de_{w,v}^{\Kr}\frac{\bfg}{{}^{(w^{-1}w_0)^\dyn}\bfh}]\\
&=\unit\de_{w,v}^{\Kr}\bfg\sum_{u_1}\frac{1}{{}^{(u_1w^{-1}w_0)^\dyn}\bfh}.
\end{align*}
Here equality $(*)$ follows from \eqref{eq:piEllopp} and the fact that the action of $Y_\Pi'$ is $\cQ$-linear. 
Rewriting $u_1w^{-1}$ as $u_1'$, we obtain the result.
\end{proof}
\begin{proof}[Proof 2]
From Corollary \ref{cor:adj} we know that 
\begin{align*}
Y_\Pi\bullet (\bfE^\la_w\cdot \Ell^{-\la}_v)&=Y_\Pi \bullet [(T_{w^{-1}}^\la\bullet f_e)\cdot (T^{-\la}_{v^{-1}w_0}\bullet f_{w_0})]=Y_\Pi\bullet [f_e\cdot (T^{-\la}_{wv^{-1}w_0}\bullet f_{w_0})].
\end{align*}
Since $f_ef_u=\de_{e,u}^{\Kr}$, and by using the formula of the $\bullet$- action \eqref{eq:bullet}, we just need to look at the coefficient of $\de_{w_0}$ in the expansion of $T^{-\la}_{wv^{-1}w_0}=\sum_{u}a^{-\la}_{wv^{-1}w_0, u}\de^{\dyn}_{wv^{-1}w_0}\de_u$, which is 0 if $w\neq v$, and is $ a^{-\la}_{w_0,w_0}\de_{w_0}^{\dyn}=\frac{\bfg}{{}^{w_0^\dyn}\bfh}\de_{w_0}^{\dyn}$ if $w=v$. So if $w\neq v$, we see that $Y_\Pi(\bfE_w^\la\cdot \Ell^{-\la}_v)=0$. If $w=v$, we have 
\begin{align*}
Y_\Pi(\bfE_w^{\la}\cdot \Ell_w^{-\la})&=Y_\Pi\bullet (f_e\cdot (a^{-\la}_{w_0,w_0}\de_{w_0}^{\dyn}\de_{w_0}\bullet f_{w_0}))\\
&=Y_\Pi\bullet (f_e\frac{\bfg}{{}^{w_0^\dyn}\bfh})\\
&=(\sum_{u\in W}\de_u\sum_{u_1\in W}\de_{u_1}^{\dyn})\bullet f_e \frac{1}{{}^{w_0^\dyn}\bfh}\\
&=\sum_{u\in W}f_{u^{-1}}\sum_{u_1\in W}{}^{u_1^{\dyn}}(\frac{1}{{}^{w_0^\dyn}\bfh})\\
&=\unit\sum_{u_1\in W}\frac{1}{{}^{u_1^{\dyn}}\bfh}. 
\end{align*}
\end{proof}

\bibliographystyle{amsplain}

\begin{thebibliography}{00}
\bibitem{AO21} M. Aganagic and A. Okounkov, {\em Elliptic stable envelopes}, J. Amer. Math. Soc. 34 (2021), no. 1, 79–133.

\bibitem{CZZ15} B. Calm\`es, K. Zainoulline, C. Zhong,  {\em Equivariant oriented cohomology of flag varieties}, Doc. Math., Extra Volume: Alexander S. Merkurjev's Sixtieth Birthday (2015), 113-144.


\bibitem{KK86} B. Kostant, S. Kumar, {\em The nil Hecke ring and cohomology of  $G/P$  for a Kac-Moody group}, Adv. in Math. 62 (1986), no. 3, 187–237.


\bibitem{KK90} B. Kostant, S. Kumar, {\em T-equivariant  K-theory of generalized flag varieties}, J. Differential Geom. 32 (1990), no. 2, 549–603. 


\bibitem{KRW20} S. Kumar, R. Rim\'anyi, A. Weber, {\em Elliptic classes of Schubert varieties}, Math. Ann. 378 (2020), no. 1-2, 703–728.

\bibitem{LZZ23} C. Lenart, G. Zhao, C. Zhong, {\em  Elliptic classes via the periodic Hecke module and its Langlands dual}, Preprint, arXiv:2309.09140. 


\bibitem{RW19} R. Rim\'anyi, A. Weber, {\em Elliptic classes of Schubert varieties via Bott-Samelson resolution}, J. Topol. 13 (2020), no. 3, 1139–1182.



\end{thebibliography}

\end{document}